\documentclass[11pt]{article}
\usepackage{amsmath}
\usepackage{dsfont}
\usepackage{amsmath}
\usepackage{amsmath,bm}
\usepackage{mathrsfs}
\usepackage{amsmath,amssymb}
\usepackage{amsfonts}
\usepackage{hyperref}
\usepackage{amsthm}
\usepackage{graphicx}
\usepackage{subfigure}
\usepackage{xcolor}
\usepackage{appendix}
\usepackage{cite}

\usepackage{tikz}

\usetikzlibrary{calc,shapes,arrows}

\usepackage{pgfplots}

\usetikzlibrary{patterns}

\usepgfplotslibrary{dateplot}

\renewcommand{\qed}{\hfill\small{$\square$}\normalsize}

\hfuzz=\maxdimen
\theoremstyle{definition}
\newtheorem{lemma}{Lemma}[section]
\newtheorem{definition}{Definition}

\newtheorem{theorem}[lemma]{Theorem}
\newtheorem{corollary}[lemma]{Corollary}

\newtheorem{remark}{Remark}

\numberwithin{equation}{section}

\renewcommand{\qed}{\hfill\small{$\square$}\normalsize}

\DeclareFixedFont{\Acknowledgment}{OT1}{cmr}{bx}{n}{14pt}
\begin{document}

\title{\bf Combinatorial Calabi flows on surfaces with boundary}
\author{Yanwen Luo, Xu Xu}
\maketitle

\begin{abstract}
Motivated by Luo's combinatorial Yamabe flow on closed surfaces \cite{L1} and
Guo's combinatorial Yamabe flow on surfaces with boundary \cite{Guo},
we introduce combinatorial Calabi flow on ideally triangulated surfaces with boundary, aiming at finding
hyperbolic metrics on surfaces with totally geodesic boundaries of given lengths.
Then we prove the long time existence and global convergence of combinatorial Calabi flow on surfaces with boundary.
We further introduce fractional combinatorial Calabi flow on surfaces with boundary,
which unifies and generalizes the combinatorial Yamabe flow and the combinatorial Calabi flow on surfaces with boundary.
The long time existence and global convergence of fractional combinatorial Calabi flow are also proved.
These combinatorial curvature flows provide effective algorithms to
construct hyperbolic surfaces with totally geodesic boundaries with prescribed lengths.
\end{abstract}



\maketitle

\section{Introduction}

The resolution of the Poincar\'{e} Conjecture by Ricci flow marks the culmination of the research of curvature flows on manifolds. It led to a flurry of study on various types of geometric flows to deform Riemannian metrics on a manifold.
The purpose of this paper is to extend this study to surfaces with boundary in the combinatorial setting.
We introduce \textit{combinatorial Calabi flows} on hyperbolic surfaces with boundary equipped with ideal triangulations,
which deform a hyperbolic metric on a surface with boundary
within the discrete conformal class of the metric.
We prove that the combinatorial Calabi flows converge to hyperbolic surfaces with totally geodesic boundary, whose lengths can be arbitrarily prescribed.

The notion of discrete conformal structures on triangulated surfaces is a discrete analogue of their smooth counterparts, conformal structures on smooth surfaces. One motivation to study discrete conformal structures is to compute the conformal maps between planar regions in applications.  Thurston \cite{T1} rediscovered the circle packing theorem initiated by Koebe \cite{K1} and Andreev \cite{Andr1,Andr2},
and proposed circle packings as a natural discretization of conformal maps. This idea was carried out by Rodin-Sullivan \cite{RS}. Since then, different types of discrete conformal structures on closed surfaces have been extensively studied in the last two decades, including  tangential circle packings, Thurston's circle packings, inversive distance circle packings, Luo's vertex scaling, mixed types of discrete conformal structures, etc.
See, for instance, \cite{BS, CL, G3, GT, GL,L1, L3,  X1, X2, ZGZLYG} and others.
To find circle packings with prescribed combinatorial curvatures on closed triangulated surfaces,
Chow-Luo \cite{CL} introduced the combinatorial Ricci flow for Thurston's circle packing metrics on closed surfaces.
Luo \cite{L1} further introduced a new type of discrete conformal structure called vertex scaling for piecewise flat metrics on closed surfaces,
and the corresponding combinatorial Yamabe flow deforming discrete metrics within a discrete conformal class.
This new notion of discrete conformal structure leads to rigidity results of polyhedral surfaces with respect to discrete curvature \cite{BPS} and discrete uniformization theorems on closed polyhedral surfaces \cite{GGLSW, GLSW, Sp}.
These works give rise to tons of new results on combinatorial curvature flows and discrete conformal structures on closed surfaces \cite{GLW, LSW, WZ, WGS, W}, and lead to various applications in surface matching, surface parametrization, manifold spline and others.
See \cite{SWGL, ZG} for a comprehensive survey on this topic.

In contrast to the case of closed surfaces, there are only sporadic results on discrete conformal structures and discrete curvature flows on surfaces with boundary.
The first notion of discrete conformal structure, also called vertex scaling,
was introduced by Guo in \cite{Guo} for hyperbolic metrics on ideally triangulated surfaces with boundary.
Guo \cite{Guo} further introduced the corresponding combinatorial Yamabe flow and studied its properties. See also \cite{X3}.
In this paper, we introduce combinatorial Calabi flows for Guo's vertex scaling and study their long time behavior.

\subsection{Set up}
Suppose $\Sigma$ is a compact surface with boundary $B$ consisting of $n$ connected components, which are topologically circles.
An ideal triangulation $\mathcal{T}$ of $\Sigma$ could be constructed using topological hexagons as follows.
Suppose we have a disjoint union of finite colored hexagons such that three non-adjacent edges of each hexagon are colored red and
the other three edges colored black.
Then we identify the red edges of the colored hexagons in pairs by homeomorphisms.
The quotient space of these colored hexagons is an ideally triangulated compact surface with boundary.
For simplicity, we define a face in the ideal triangulation $\mathcal{T}$ as the image of one hexagon under the quotient map.
Similarly, an edge in $\mathcal{T}$ is the image of one red edge in the hexagons under the quotient map, and a boundary component is the quotient of a sequence of black edges glued together at their endpoints.
We will denote the connected components of the boundary $B$ as $\{1,2, \cdots, n\}$, and the set of edges and faces as $E$ and $F$ respectively.

A well-known fact in hyperbolic geometry is that for any three positive numbers, there exists a unique hyperbolic righted-angled hexagon up to isometry, where the lengths of three non-adjacent edges are given by the three positive numbers \cite{Rat}.
If each edge of an ideal triangulation $\mathcal{T}$ of a compact surface with boundary $\Sigma$ is assigned with a positive number, then each face in $\mathcal{T}$ can be realized as a hyperbolic right-angled hexagon.
Gluing these geometric hexagons along the edges in pairs by hyperbolic isometries, we can construct
a hyperbolic surface with totally geodesic boundary $\Sigma$ from the ideal triangulation $\mathcal{T}$. Conversely, any ideally triangulated hyperbolic surface with totally geodesic boundary $(\Sigma, \mathcal{T})$ produces a function $l: E\rightarrow (0, +\infty)$ sending each edge to the length of the unique geodesic orthogonal to boundary geodesics in its homotopy class.
The function $l: E\rightarrow (0, +\infty)$ is called a \emph{discrete hyperbolic metric} on $(\Sigma, \mathcal{T})$. Moreover, the lengths of its boundary components define a function $K: B\to (0 + \infty)$, called the \emph{generalized combinatorial curvature}.

Motivated by Luo's definition of vertex scaling of piecewise flat metrics on triangulated closed surfaces \cite{L1},
Guo \cite{Guo} introduced the following definition of discrete conformality for
discrete hyperbolic metrics on ideally triangulated surfaces with boundary, called vertex scaling as well.

\begin{definition}[Guo \cite{Guo}]\label{defn Guo's vertex scaling}
Suppose $(\Sigma, \mathcal{T})$ is an ideally triangulated surface with boundary.
Let $l$ and $\tilde{l}$ be two discrete hyperbolic metrics on $(\Sigma, \mathcal{T})$.
If there exists a function $w: B\rightarrow \mathbb{R}$ such that
\begin{equation}\label{Guo's vertex scaling}
  \cosh \frac{l_{ij}}{2}=e^{w_i+w_j}\cosh \frac{\tilde{l}_{ij}}{2},
\end{equation}
then the discrete hyperbolic metric $l$ is called vertex scaling of $\tilde{l}$. The function $w: B\rightarrow \mathbb{R}$ is called  a discrete conformal factor.
\end{definition}

For a fixed discrete hyperbolic metric $\tilde{l}$, a discrete conformal factor $w: B\rightarrow \mathbb{R}$ is \textit{admissible} if (\ref{Guo's vertex scaling}) defines a discrete hyperbolic metric $l: E\rightarrow (0, +\infty)$.
The space of admissible discrete conformal factors of a discrete hyperbolic metric $\tilde{l}$ is denoted by $\mathcal{W} = \mathcal{W}(\tilde{l})$.
The discrete hyperbolic metric $l$ in  (\ref{Guo's vertex scaling}) will be referred as $w*\tilde{l}$ for simplicity.

Guo \cite{Guo} further introduced the following combinatorial Yamabe flow
\begin{equation}\label{Guo's CYF}
  \begin{aligned}
  \left\{
  \begin{array}{ll}
    \frac{dw_i(t)}{dt}=K_i(t), & \hbox{ } \\
    w_i(0)=0, & \hbox{ }
  \end{array}
\right.
  \end{aligned}
\end{equation}
and proved its long time existence and global convergence to a cusped hyperbolic surface.
Motivated by Guo's combinatorial Yamabe flow (\ref{Guo's CYF}),
the second author \cite{X3} introduced the following combinatorial Yamabe flow
\begin{equation}\label{Xu's CYF}
\begin{aligned}
\left\{
  \begin{array}{ll}
    \frac{dw_i(t)}{dt}=K_i(t)-\overline{K}_i,  &\hbox{ } \\
    w(0)=w_0,  &\hbox{ }
  \end{array}
\right.
\end{aligned}
\end{equation}
where $\overline{K}: B\rightarrow (0, +\infty)$ is a function defined on the boundary components and $w_0\in \mathcal{W}$.
The combinatorial Yamabe flow (\ref{Xu's CYF}) aims at finding hyperbolic metrics on $\Sigma$ with a prescribed generalized combinatorial curvature $\overline{K}: B\rightarrow (0, +\infty)$, equivalently, prescribed lengths of connected components of $B$.
The long time existence and global convergence of the combinatorial Yamabe flow (\ref{Xu's CYF}) were proved in \cite{X3}.

\subsection{Main Results}
In this paper, we introduce the following combinatorial Calabi flow for Guo's vertex scaling of
discrete hyperbolic metrics on
ideally triangulated surfaces with boundary.
\begin{definition}
Assume $(\Sigma, \mathcal{T})$ is a surface with boundary with an ideal triangulation.
Let $\tilde{l}\in \mathbb{R}^E_{>0}$ be a discrete hyperbolic metric defined on $(\Sigma, \mathcal{T})$ and
$\bar{K}\in \mathbb{R}^n_{>0}$ be a function defined on the boundary components $B=\{1,2,\cdots, n\}$.
The \textit{combinatorial Calabi flow} is defined to be
\begin{equation}\label{CCF}
\begin{aligned}
\left\{
  \begin{array}{ll}
    \frac{dw_i(t)}{dt} =  \Delta (\bar{K}-K )_i, & \hbox{ } \\
    w(0) = w_0, & \hbox{ }
  \end{array}
\right.
\end{aligned}
\end{equation}
where $K_i(t)$ is the generalized combinatorial curvature of
the discrete hyperbolic metric $l(t) = w(t)* \tilde{l}$,
$\Delta= (\frac{\partial K_i}{\partial w_j})_{n\times n}$ is the discrete Laplace operator defined by the Jacobian of $K(t)$ with respect to the discrete conformal factor $w(t)$, and $w_0\in \mathcal{W}= \mathcal{W}(\tilde{l})$ is an admissible discrete conformal factor.
\end{definition}

The combinatorial Calabi flow was first introduced by Ge in \cite{Ge-thesis} (see also \cite{Ge}) for Thurston's Euclidean circle packing metrics on closed triangulated surfaces and then further studied for different discrete metric structures on closed surfaces and 3-dimensional manifolds.
See, for instance, \cite{GH1,GX,WX, X, ZX} and others. This is the first time that the combinatorial Calabi flow is introduced on ideally triangulated surfaces with boundary.

We prove the following result on the long time existence and global convergence of combinatorial Calabi flows (\ref{CCF})
on ideally triangulated hyperbolic surfaces with boundary.

\begin{theorem} \label{main theorem on CCF}
Assume $(\Sigma, \mathcal{T}, \tilde{l})$ is an ideally triangulated surface with boundary with a discrete hyperbolic metric $\tilde{l} \in \mathbb{R}^E_{>0}$ and $\bar{K}\in \mathbb{R}^n_{>0}$ is a function defined on the boundary components $B=\{1,2,\cdots, n\}$.
Then the solution $w(t)$ to the combinatorial Calabi flow (\ref{CCF})  
exists for all time and converges exponentially fast 
for any initial $w_0\in \mathcal{W}$.
\end{theorem}

In \cite{Guo}, Guo proved the following basic properties of the vertex scaling in Definition \ref{defn Guo's vertex scaling}.

\begin{theorem}[Guo \cite{Guo}]\label{negative definiteness of Laplacian}
Assume $(\Sigma, \mathcal{T}, \tilde{l})$ is an ideally triangulated surface with boundary with a discrete hyperbolic metric $\tilde{l} \in \mathbb{R}^E_{>0}$.
Then
\begin{itemize}
  \item[(1)] the admissible space $\mathcal{W}$ is a convex polytope;
  \item[(2)] the discrete Laplace operator $\Delta=(\frac{\partial K_i}{\partial w_j})_{n\times n}$ is symmetric and strictly negative definite on the admissible space $\mathcal{W}$.
  \item[(3)] The generalized combinatorial curvature map $K: \mathcal{W}\rightarrow \mathbb{R}^n_{>0}$ is a diffeomorphism. 
\end{itemize}
\end{theorem}

Based on Theorem \ref{negative definiteness of Laplacian}, we can define the fractional discrete Laplace operator $\Delta^s$ for Guo's vertex scaling of discrete hyperbolic metrics on ideally triangulated surfaces with boundary for any $s\in \mathbb{R}$ as follows. Recall that if $A$ is a symmetric positive definite $n\times n$ matrix and $P\in O(n)$ is  an orthogonal matrix with
\begin{equation*}
  \begin{aligned}
A=P^T\cdot \text{diag}\{\lambda_1, \cdots, \lambda_n\}\cdot P,
\end{aligned}
\end{equation*}
where $\lambda_1, \cdots, \lambda_n$ are positive eigenvalues of the matrix $A$.
Then $A^s$ is defined to be
\begin{equation*}
  \begin{aligned}
A^s=P^T\cdot \text{diag}\{\lambda_1^s, \cdots, \lambda_n^s\}\cdot P.
\end{aligned}
\end{equation*}
The $2s$-th order fractional discrete Laplace operator $\Delta^s$ for Guo's vertex scaling of discrete hyperbolic metrics on ideally triangulated surfaces with boundary is defined to be
\begin{equation}\label{fractional Laplacian}
  \begin{aligned}
\Delta^s=-(-\Delta)^s,
\end{aligned}
\end{equation}
where $\Delta= (\frac{\partial K_i}{\partial w_j})_{n\times n}$ is the standard discrete Laplace operator in Theorem \ref{negative definiteness of Laplacian}.
If $s=0$, the fractional discrete Laplace operator $\Delta^s$ is reduced to the minus identity operator.
If $s=1$, the fractional discrete Laplace operator $\Delta^s$ is reduced to the discrete Laplace operator $\Delta= (\frac{\partial K_i}{\partial w_j})_{n\times n}$.
By Theorem \ref{negative definiteness of Laplacian},
the fractional discrete Laplace operator $\Delta^s$ is strictly negative definite on the admissible space $\mathcal{W}$ for any $s\in \mathbb{R}$.

Following Wu-Xu \cite{WX},
we introduce the following fractional combinatorial Calabi flow for
Guo's vertex scaling of discrete hyperbolic metrics on ideally triangulated surfaces with boundary.

\begin{definition}
Assume $(\Sigma, \mathcal{T})$ is an ideally triangulated surface with boundary.
Let $s\in \mathbb{R}$, $\tilde{l}\in \mathbb{R}^E_{>0}$ be a discrete hyperbolic metric defined on $(\Sigma, \mathcal{T})$ and
$\bar{K}\in \mathbb{R}^n_{>0}$ be a function defined on the boundary components $B=\{1,2,\cdots, n\}$.
The \textit{fractional combinatorial Calabi flow} is defined to be
\begin{equation}\label{FCCF}
\begin{aligned}
\left\{
  \begin{array}{ll}
    \frac{dw_i(t)}{dt} =  \Delta^s (\bar{K}-K )_i, & \hbox{ } \\
    w(0) = w_0, & \hbox{ }
  \end{array}
\right.
\end{aligned}
\end{equation}
where $K_i(t)$ is the generalized combinatorial curvature of the discrete hyperbolic metric $w(t)* \tilde{l}$,
$\Delta^s$ is the fractional discrete Laplace operator of $w(t)* \tilde{l}$ defined by (\ref{fractional Laplacian})
and $w_0\in \mathcal{W}= \mathcal{W}(\tilde{l})$ is an admissible discrete conformal factor.
\end{definition}

If $s=0$, the fractional combinatorial Calabi flow (\ref{FCCF}) is reduced to the combinatorial Yamabe flow (\ref{Xu's CYF}).
If $s=1$, the fractional combinatorial Calabi flow (\ref{FCCF}) is reduced to  the combinatorial Calabi flow (\ref{CCF}).
The fractional combinatorial Calabi flow (\ref{FCCF}) further covers the case of $s\neq 0, 1$.
Note that the fractional combinatorial Calabi flow (\ref{FCCF}) is in general a non-local combinatorial curvature flow.
We prove the following result on long time existence and global convergence of the fractional combinatorial Calabi flow (\ref{FCCF}).

\begin{theorem} \label{main theorem on FCCF}
Assume $(\Sigma, \mathcal{T})$ is an ideally triangulated surface with boundary.
Let $\tilde{l}\in \mathbb{R}^E_{>0}$ be a discrete hyperbolic metric defined on $(\Sigma, \mathcal{T})$ and
$\bar{K}\in \mathbb{R}^n_{>0}$ be a function defined on the boundary components $B=\{1,2,\cdots, n\}$.
Then the solution  $w(t)$ to the fractional combinatorial Calabi flow (\ref{FCCF})  exists for all time and converges exponentially fast for any initial $w_0\in \mathcal{W}$.
\end{theorem}

\subsection{Organization of the paper} The paper is organized as follows. In Section \ref{section 2}, we study the basic properties of the combinatorial Calabi flow (\ref{CCF})
and prove Theorem \ref{main theorem on CCF}.
In Section \ref{section 3}, we summarize the basic properties of the fractional combinatorial Calabi flow (\ref{FCCF})
and prove Theorem \ref{main theorem on FCCF}.
\\
\\
\textbf{Acknowledgements}\\[8pt]
The authors thank Professor Ren Guo for his constant encouragements and supports
and thank Professor Ze Zhou for valuable communications.
The research of the second author is supported by the Fundamental Research Funds for the Central Universities under Grant no. 2042020kf0199.

\section{Combinatorial Calabi flow on surfaces with boundary}\label{section 2}
In this section, we prove the main result of this paper about the convergence of combinatorial Calabi flow on ideally triangulated surface with boundary.
\subsection{Basic properties of combinatorial Calabi flow}
\begin{lemma}\label{CCF is a gradient flow}
The combinatorial Calabi flow (\ref{CCF}) is a negative gradient flow of the  combinatorial Calabi energy defined by
$$\mathcal{C}(w)  = \frac{1}{2}\sum_{i = 1}^n(K_i(w) - \bar{K}_i)^2.$$
\end{lemma}
\begin{proof}
By direct calculations, we have
$$\frac{\partial \mathcal{C}}{\partial w_j}= \sum_{i = 1}^n \frac{\partial K_i}{\partial w_j}(K_i - \bar{K}_i) =  \Delta (K - \bar{K})_j = -\frac{dw_j}{dt},$$
which implies that the combinatorial Calabi flow (\ref{CCF}) is a negative gradient flow of the combinatorial Calabi energy
$\mathcal{C}(w(t))$.
\qed
\end{proof}

As a corollary, we have the following result.
\begin{corollary}\label{decreasing of Calabi energy}
  The combinatorial Calabi energy $\mathcal{C}(w)$ is decreasing along the combinatorial Calabi flow (\ref{CCF}).
\end{corollary}
\begin{proof}
By direct calculations, we have
$$\frac{d\mathcal{C}(w(t))}{dt} = \sum_{i=1}^n \frac{\partial \mathcal{C}}{\partial w_i }\frac{dw_i}{dt} =  -\sum_{i=1}^n (\Delta (K - \bar{K})_i)^2\leq 0,$$
the right side of which is strictly negative unless $K = \bar{K}$.
\qed
\end{proof}

By Theorem \ref{negative definiteness of Laplacian}, the following function
\begin{equation}\label{energy function E}
\mathcal{E}(w) = - \int_0^w \sum_{i = 1}^n (K_i - \bar{K}_i) dw_i
\end{equation}
is a well-defined smooth convex function defined on the admissible space $\mathcal{W}$.
Guo \cite{Guo} first proved that the energy function $\mathcal{E}(w)$ with $\bar{K}=0$
is a strictly convex energy function on the admissible space $\mathcal{W}$.
The global rigidity of Guo's vertex scaling of
discrete hyperbolic metrics with respect to the lengths of the boundary components in \cite{Guo} follows from the convexity of this function.

The following monotonicity property of the energy function $\mathcal{E}(w)$ holds along the combinatorial Calabi flow (\ref{CCF}).

\begin{lemma}\label{decreasing of E}
\label{decrease}
The energy function $\mathcal{E}(w)$ defined by (\ref{energy function E}) is decreasing along the combinatorial Calabi flow (\ref{CCF}).
\end{lemma}
\begin{proof}
By direct calculations, we have
$$\frac{d\mathcal{E}}{dt} = \sum_{i = 1}^n \frac{\partial \mathcal{E}}{\partial w_i} \frac{dw_i}{dt}
= \sum_{i = 1}^n (K - \bar{K})_i \Delta (K - \bar{K})_i =(K - \bar{K})^T\Delta (K - \bar{K})\leq 0,$$
where the last inequality follows from Guo's Theorem \ref{negative definiteness of Laplacian} on the negative definiteness of
the discrete Laplace operator $\Delta$.
\qed
\end{proof}

\subsection{The long time behavior of combinatorial Calabi flow}
As the combinatorial Calabi flow (\ref{CCF}) is essentially a system of ordinary differential equations, the solution to the combinatorial Calabi flow (\ref{CCF})
exists locally by the standard theory in dynamical system.
The focus of this subsection is on the long time behavior of the combinatorial Calabi flow (\ref{CCF}).

\begin{lemma}
If the solution $w(t)$ to the combinatorial Calabi flow (\ref{CCF}) exists for all time and
converges to $\bar{w}\in \mathcal{W}$, then $K(\bar{w})=\bar{K}$.
\end{lemma}
\begin{proof}
  By the proof of  Corollary \ref{decreasing of Calabi energy}, we have
$$\frac{d\mathcal{C}(w(t))}{dt} =  - \sum_{i = 1}^n (\Delta (K - \bar{K})_i)^2\leq 0,$$
which implies that the combinatorial Calabi energy $\mathcal{C}(w)$ is decreasing along the combinatorial Calabi flow (\ref{CCF}).
Note that $\mathcal{C}(w)\geq 0$ by definition, and $\lim_{t\rightarrow +\infty}\mathcal{C}(w(t))$ exists.
As a result, there exists $\xi_n\in (n,  n+1)$ such that
\begin{equation*}
\mathcal{C}(w(n+1))-\mathcal{C}(w(n))=\frac{d}{dt}\Big|_{t=\xi_n}\mathcal{C}(w(t))=
- \sum_{i = 1}^n \big(\Delta (K(w(t)) - \bar{K})_i\big)^2|_{t=\xi_n}\rightarrow 0,\ \text{as}\ n\rightarrow +\infty.
\end{equation*}
By the continuity of the generalized combinatorial curvature $K$ in the discrete conformal factor $w$, we have
$$\Delta (K(w)-\bar{K})|_{w=\bar{w}}=0$$
by the assumption $\lim_{t\rightarrow +\infty}w(t)=\bar{w}$.
By the strictly negative definiteness of the discrete Laplace operator $\Delta$ in Theorem \ref{negative definiteness of Laplacian}, we have
$K(\bar{w})=\bar{K}.$
\qed
\end{proof}

\begin{lemma}\label{w(t) is bounded along CCF}
Suppose $\bar{K}\in \mathbb{R}^n_{>0}$ is a function defined on the boundary components $B=\{1,2,\cdots, n\}$.
Then for any initial value $w(0) = w_0\in \mathcal{W}$,
the solution $w(t)$ to the combinatorial Calabi flow (\ref{CCF})
stays in a bounded  subset of the admissible space $\mathcal{W}$.
\end{lemma}
\begin{proof}
By Theorem \ref{negative definiteness of Laplacian}, for any $\bar{K}\in \mathbb{R}^n_{>0}$,
there exists 
$\bar{w}\in \mathcal{W}$ such that $\bar{K}=K(\overline{w})$, which implies
$$\nabla \mathcal{E}(\bar{w}) = - (K(\bar{w}) - \bar{K})=0.$$
Note that $\mathcal{E}(w)$ is a strictly convex function on $\mathcal{W} \subset \mathbb{R}^n$,
we have
\begin{equation}\label{bdd proof eqn 1}
\lim_{w\to \infty, w\in \mathcal{W}} \mathcal{E}(w) = +\infty.
\end{equation}
Since $\mathcal{E}(w)$ is decreasing along the combinatorial Calabi flow (\ref{CCF}) by Lemma \ref{decreasing of E}, we have
\begin{equation}\label{bdd proof eqn 2}
\mathcal{E}(w(t)) \leq \mathcal{E}(w(0)).
\end{equation}
Combining (\ref{bdd proof eqn 1}) and (\ref{bdd proof eqn 2}), we prove that
the solution $w(t)$ to the combinatorial Calabi flow (\ref{CCF}) stays in a bounded subset of $\mathbb{R}^n$ along the flow.
\qed
\end{proof}

To prove the long time existence of the solution $w(t)$ to combinatorial Calabi flow (\ref{CCF}),
we need to further prove that the solution $w(t)$ stays in a compact subset of the
admissible space $\mathcal{W}$.
Recall the following characterization of the admissible space $\mathcal{W}$ obtained by Guo \cite{Guo}.

\begin{lemma}[Guo \cite{Guo}]\label{Guo's thm on admissible space}
The admissible space
\begin{equation}\label{admissible space}
  \mathcal{W}=\cap_{\{ij\}\in E}\mathcal{W}_{ij},
\end{equation}
where
$$\mathcal{W}_{ij}=\{w\in \mathbb{R}^n| w_i+w_j>-\ln \cosh \frac{\tilde{l}_{ij}}{2}\}.$$
\end{lemma}

The boundary of each $\mathcal{W}_{ij}$ in $\partial\mathcal{W}$ is a portion of the hyperplane induced by $l_{ij} = 0$, namely,
$$\partial_{ij} \mathcal{W} = \{w\in \partial \mathcal{W}  | w\in \mathbb{R}^n,  w_i+w_j = -\ln \cosh \frac{\tilde{l}_{ij}}{2} \}.$$
Note that any point of $\partial \mathcal{W}$ lies on some $\partial_{ij}\mathcal{W}$.
We will prove that the solution $w(t)$ to the combinatorial Calabi flow (\ref{CCF})
will never reach the boundary $\partial\mathcal{W}$.
First, we have the following result.
\begin{lemma}[\cite{Guo, X3}]\label{proper}
For any $M>0$, there exists a positive constant $\epsilon_{ij} = \epsilon_{ij}(M)$ such that if $w\in \mathcal{W}$ satisfies
$$w_i + w_j < - \ln \cosh \frac{\tilde{l}_{ij}}{2} + \epsilon_{ij},$$
then
$$K_i>M, K_j>M. $$
\end{lemma}
\begin{proof}
For completeness, we sketch the proof here.
For a hyperbolic right-angled hexagon $\{ijk\}\in F$ adjacent to $i,j\in B$, we have
\begin{equation}\label{theta tends infinity uniformly}
\begin{aligned}
\cosh \theta_i^{jk}
=\frac{\cosh l_{ij}\cosh l_{ik}+\cosh l_{jk}}{\sinh l_{ij}\sinh l_{ik}}
>\frac{\cosh l_{ij}\cosh l_{ik}}{\sinh l_{ij}\sinh l_{ik}}
>\frac{\cosh l_{ij}}{\sinh l_{ij}},
\end{aligned}
\end{equation}
by the cosine law for hyperbolic right-angled hexagons, where $\theta_i^{jk}$ is the length of the hyperbolic arc at $i\in B$ in the hyperbolic right-angled hexagon $\{ijk\}\in F$.
The formula (\ref{theta tends infinity uniformly}) implies that $\theta_i^{jk}\rightarrow +\infty$ uniformly as $l_{ij}\rightarrow 0^+$,
which further implies that $K_i\rightarrow +\infty$ uniformly as $w\rightarrow \partial_{ij} \mathcal{W}$.
The same arguments apply to $K_j$.
\qed
\end{proof}
As an application of Lemma \ref{proper}, we have the following result.
\begin{lemma}\label{w(t) can not in W espsilon}
Assume $(\Sigma, \mathcal{T})$ is an ideally triangulated surface with boundary.
Let $\tilde{l}\in \mathbb{R}^E_{>0}$ be a discrete hyperbolic metric defined on $(\Sigma, \mathcal{T})$ and
$\bar{K}\in \mathbb{R}^n_{>0}$ be a function defined on the boundary components $B$.
For any initial value $w_0\in \mathcal{W}$,
there exists a constant $\epsilon = \epsilon(w_0, \bar{K}) > 0$
such that the solution $w(t)$ to the combinatorial Calabi flow (\ref{CCF}) can never be in the region
$$\mathcal{W}_\epsilon = \{w\in \mathcal{W}|d(w, \partial \mathcal{W})< \epsilon\},$$
where $d$ is the standard Euclidean metric on $\mathbb{R}^n$.
\end{lemma}

\begin{proof}
Set $$M = \max_{i} \{|\bar{K}_i|+ \sqrt{2\mathcal{C}(w_0)}\}.$$
By Lemma \ref{proper},
there exists $\epsilon_{ij}=\epsilon_{ij}(M)>0$ such that if
$$w_i+w_j < -\ln \cosh \frac{\tilde{l}_{ij}}{2} +2 \epsilon_{ij},$$
then
 $$K_i(w)>M, K_j(w)>M.$$
Set $\epsilon_0=\min_{\{ij\}\in E} \epsilon_{ij}>0$, then if $w\in \mathcal{W}$ satisfies
$$w_i+w_j < -\ln \cosh \frac{\tilde{l}_{ij}}{2} +2 \epsilon_0$$
for some edge $\{ij\}\in E$, then
$K_i(w)>M,$
which further implies that
\begin{equation}\label{equ in proof of solu in W}
  |K_i(w) - \bar{K}_i| \geq |K_i(w)| - |\bar{K}_i| > M - |\bar{K}_i|\geq\sqrt{2\mathcal{C}(w_0)}.
\end{equation}

We claim that the solution $w(t)$ to the combinatorial Calabi flow (\ref{CCF})  can never be in the region $\mathcal{W}_{\epsilon_0}$.
Otherwise, there exists some $t_0\in [0, +\infty)$ and an edge $\{ij\} \in E$ such that the solution $w(t)$ to the combinatorial Calabi flow (\ref{CCF})
satisfies $w(t_0)\in \mathcal{W}$ and
$$w_i(t_0)+w_j(t_0) < -\ln \cosh \frac{\tilde{l}_{ij}}{2} +2 \epsilon_0,$$
which implies that
\begin{equation}\label{equ 1 in proof of compactnees}
  |K_i(w(t_0))- \bar{K}_i|>\sqrt{2\mathcal{C}(w_0)}
\end{equation}
by (\ref{equ in proof of solu in W}).
Note that $\mathcal{C}(w)$ is decreasing along the combinatorial Calabi flow (\ref{CCF}) by Corollary \ref{decreasing of Calabi energy}.
Therefore,  for any $t>0$, the solution $w(t)$ to the combinatorial Calabi flow (\ref{CCF}) satisfies
$$|K_i(t) - \bar{K}_i| \leq \sqrt{2\mathcal{C}(w(t))} \leq \sqrt{2\mathcal{C}(w_0)}$$
for any $i\in B$, which contradicts to (\ref{equ 1 in proof of compactnees}).
Therefore, the solution $w(t)$ to the combinatorial Calabi flow (\ref{CCF}) can never be in the region $\mathcal{W}_{\epsilon_0}$.
\qed
\end{proof}

\begin{remark}
  The result in Lemma \ref{w(t) can not in W espsilon} is independent of the existence of $\bar{w}\in \mathcal{W}$
with $K(\bar{w})=\bar{K}$.
\end{remark}

As a direct corollary of Lemma \ref{w(t) is bounded along CCF} and Lemma \ref{w(t) can not in W espsilon},
we have the following result on the solution to the combinatorial Calabi flow (\ref{CCF}),
which implies the long time existence of the solution to the combinatorial Calabi flow (\ref{CCF}).

\begin{corollary}\label{compactness of the solution}
Suppose $(\Sigma, \mathcal{T}, \tilde{l})$ is an ideally triangulated compact surface with boundary with a discrete hyperbolic metric $\tilde{l} \in \mathbb{R}^E_{>0}$ and $\bar{K}\in \mathbb{R}^n_{>0}$ is a function defined on $B$.
Then for any initial value $w_0\in \mathcal{W}$,
the solution $w(t)$ to the combinatorial Calabi flow (\ref{CCF}) stays in a compact subset $\Omega$ of the admissible space $\mathcal{W}$.
As a result, the solution $w(t)$ to the combinatorial Calabi flow (\ref{CCF}) exists for all time.
\end{corollary}

\textbf{Proof of Theorem \ref{main theorem on CCF}:}
By Corollory \ref{compactness of the solution}, the solution $w(t)$ to
the combinatorial Calabi flow (\ref{CCF}) stays in a compact subset $\Omega$ of the admissible space $\mathcal{W}$.
Note that the discrete Laplace operator $\Delta$ is strictly negative by Theorem \ref{negative definiteness of Laplacian}.
By the continuity of the eigenvalue $\lambda_\Delta$  of $\Delta$, there exists a uniform positive constant $\lambda_0$ such that
$$\lambda_\Delta(w)\leq -\sqrt{\lambda_0},$$
when $w\in \Omega\subset\subset\mathcal{W}$.
Therefore, along the combinatorial Calabi flow (\ref{CCF}), we have
$$\frac{d\mathcal{C}(w(t))}{dt} =  - (K - \bar{K})^T\Delta^2(K - \bar{K})\leq -\lambda_0\mathcal{C}(w(t)),$$
which implies
$$\mathcal{C}(w(t)) \leq e^{-\lambda_0}\mathcal{C}(w_0).$$
This completes the proof of the exponential convergence of the solution $w(t)$ to
the combinatorial Calabi flow (\ref{CCF}) to $\bar{w}$ with $K(\overline{w})=\bar{K}$ by Theorem \ref{negative definiteness of Laplacian}.
\qed

\begin{remark}
One can also use the Lyapunov stability theorem (\cite{P} Chapter 5) to prove the exponential convergence of the solution
$w(t)$ to the combinatorial Calabi flow (\ref{CCF}) to $\bar{w}$.
\end{remark}

\section{Fractional Combinatorial Calabi flow on surfaces with boundary}\label{section 3}

Based on the property that the fractional Laplace operator $\Delta^s$ is strictly negative definite
on the admissible space $\mathcal{W}$ for any $s\in \mathbb{R}$,
the fractional combinatorial Calabi flow (\ref{FCCF}) has many properties similar to that of the combinatorial Calabi flow (\ref{CCF}).
As the proofs for these properties are almost the same as that of the combinatorial Calabi flow (\ref{CCF}),
we will only list these properties and omit the details of the proofs.

\begin{lemma}\label{decreasing of Calabi energy along FCCF}
  The combinatorial Calabi energy $\mathcal{C}(w)$ is decreasing along the fractional combinatorial Calabi flow (\ref{FCCF}).
\end{lemma}

\begin{lemma}\label{decreasing of E along FCCF}
\label{decrease}
The energy function $\mathcal{E}(w)$ defined by (\ref{energy function E}) is decreasing along the fractional combinatorial Calabi flow (\ref{FCCF}).
\end{lemma}

\begin{remark}
  Different from the combinatorial Yamabe flow  (\ref{Xu's CYF}) and the combinatorial Calabi flow (\ref{CCF}),
which corresponds to $s=0$ and $s=1$ for the fractional combinatorial Calabi flow (\ref{FCCF}) respectively,
the fractional combinatorial Calabi flow (\ref{FCCF}) is generically not a gradient flow.
As the fractional Laplace operator $\Delta^s$ is generically a non-local operator,
the fractional combinatorial Calabi flow (\ref{FCCF}) is generically
a non-local combinatorial curvature flow defined on the ideally triangulated surface with boundary $(\Sigma, \mathcal{T})$.
\end{remark}

We also have the following long time behavior of the solution $w(t)$ to the fractional combinatorial Calabi flow (\ref{FCCF}).
\begin{lemma}
If the solution $w(t)$ to the fractional combinatorial Calabi flow (\ref{FCCF}) exists for all time and
converges to $\bar{w}\in \mathcal{W}$, then $K(\bar{w})=\bar{K}$.
\end{lemma}

\begin{lemma}\label{w(t) is bounded along FCCF}
Suppose $\bar{K}\in \mathbb{R}^n_{>0}$ is a function defined on $B$.
Then for any $s\in \mathbb{R}$ and any initial value $w(0) = w_0\in \mathcal{W}$,
the solution $w(t)$ to the fractional combinatorial Calabi flow (\ref{FCCF})
stays in a bounded  subset of the admissible space $\mathcal{W}$.
\end{lemma}

\begin{lemma}\label{w(t) can not in W espsilon along FCCF}
Assume $(\Sigma, \mathcal{T})$ is an ideally triangulated surface with boundary.
Let $\tilde{l}\in \mathbb{R}^E_{>0}$ be a discrete hyperbolic metric defined on $(\Sigma, \mathcal{T})$ and
$\bar{K}\in \mathbb{R}^n_{>0}$ be a function defined on the boundary components $B$.
Then for  any $s\in \mathbb{R}$ and  any initial value $w_0\in \mathcal{W}$,
there exists a constant $\epsilon = \epsilon(w_0, \bar{K}) > 0$
such that the solution $w(t)$ to the fractional combinatorial Calabi flow (\ref{FCCF}) can never be in the region
$$\mathcal{W}_\epsilon = \{w\in \mathcal{W}|d(w, \partial \mathcal{W})< \epsilon\}.$$
\end{lemma}
As a direct corollary of Lemma \ref{w(t) is bounded along FCCF} and Lemma \ref{w(t) can not in W espsilon along FCCF},
we have the following result on the solution $w(t)$ to the fractional combinatorial Calabi flow (\ref{FCCF}), which implies
the long time existence of the solution to the fractional combinatorial Calabi flow (\ref{FCCF}).

\begin{corollary}\label{compactness of the solution2}
Suppose $(\Sigma, \mathcal{T})$ is an ideally triangulated  compact surface with boundary with a discrete hyperbolic metric $\tilde{l} \in \mathbb{R}^E_{>0}$ and  $\bar{K}\in \mathbb{R}^n_{>0}$ is a function defined on $B$.
 Then for any $s\in \mathbb{R}$ and any initial value $w_0\in \mathcal{W}$,
the solution $w(t)$ to the fractional combinatorial Calabi flow (\ref{FCCF}) stays in a compact subset $\Omega$ of the admissible space $\mathcal{W}$.
As a result, for any $s\in \mathbb{R}$ and any initial value $w_0\in \mathcal{W}$,
the solution $w(t)$ to the fractional combinatorial Calabi flow (\ref{FCCF}) exists for all time.
\end{corollary}

The proof of Theorem \ref{main theorem on FCCF} is paralleling to that of Theorem \ref{main theorem on CCF} with the discrete Laplace operator $\Delta$ replaced by the fractional discrete Laplace operator $\Delta^s$,
which is also strictly negative definite on $\mathcal{W}$ for any $s\in \mathbb{R}$.
We omit the details of the proof here.


(Yanwen Luo) Department of Mathematics, Rutgers University, New Brunswick
  NJ, 08817

E-mail: yl1594@rutgers.edu\\[2pt]

(Xu Xu) School of Mathematics and Statistics, Wuhan University, Wuhan 430072, P.R. China

E-mail: xuxu2@whu.edu.cn\\[2pt]
\end{document}